\documentclass[11pt]{amsart}
\usepackage{amssymb,mathrsfs,graphicx,enumerate,color}
\usepackage{colortbl}
\definecolor{black}{rgb}{0.0, 0.0, 0.0}
\definecolor{red}{rgb}{1.0, 0.5, 0.5}

\title[   ]{Inviscid limit to the shock waves for the fractal Burgers equation}

\author[Akopian]{Sona Akopian}
\address[Sona Akopian]{\newline Division of Applied Mathematics, \newline Brown University}
\email{sona\_akopian@brown.edu}

\author[Kang]{Moon-Jin Kang}
\address[Moon-Jin Kang]{\newline Department of Mathematic \& Research Institute of Natural Sciences, \newline Sookmyung Women's University, Seoul 140-742, Korea}
\email{moonjinkang@sookmyung.ac.kr}

\author[Vasseur]{Alexis F. Vasseur}
\address[Alexis F. Vasseur]{\newline Department of Mathematics, \newline The University of Texas at Austin, Austin, TX 78712, USA}
\email{vasseur@math.utexas.edu}

\newcommand{\beq}{\begin{equation}}
\newcommand{\eeq}{\end{equation}}

\newtheorem{theorem}{Theorem}[section]
\newtheorem{lemma}{Lemma}[section]

\newtheorem{proposition}{Proposition}[section]
\newtheorem{remark}{Remark}[section]

\newcommand{\bbr}{\mathbb R}
\newcommand{\R}{\mathbb R}

\newcommand{\eps}{\varepsilon}

\newcommand{\Seps}{S_\eps}
\newcommand{\ueps}{u_\eps}
\newcommand{\ve}{v_\eps}

\newcommand{\vphid}{\varphi_\delta}
\newcommand{\vphi}{\varphi}

\newcommand{\dx}{\mathrm{d}x}

\newcommand{\dy}{\mathrm{d}y}

\newcommand{\dt}{\mathrm{d}t}

\numberwithin{equation}{section}

\begin{document}
\bibliographystyle{plain}
\date{\today}


\thanks{\textbf{Acknowledgment.} S. Akopian and M.-J. Kang were partially supported by the NRF-2019R1C1C1009355. A. F. Vasseur was partially supported by the NSF Grant DMS 1209420. 
}

\begin{abstract} 
We show the vanishing viscosity limit to entropy shocks for the fractal Burgers equation  in one space dimension. 
More precisely, we quantify the rate of convergence of the inviscid limit in $L^2$ for large initial perturbations around the entropy shock on any bounded time interval. 
This is the first result on the inviscid limit to entropy shock for the fractal Burgers equation with the quantified convergence, for large initial perturbations. 
\end{abstract}
\maketitle \centerline{\date}


\section{Introduction and main results}
\setcounter{equation}{0}

We consider the Burgers equation with the fractional Laplacian in one space dimension:
\beq\label{main_1}
      \partial_t u + \partial_x \left(\frac{u^2}{2}\right) = \Delta_x^{\alpha/2}u , \quad t>0,~x\in \R,
\eeq
where $\alpha$ denotes the fractional power of the Laplacian in one dimension, and the fractional Laplacian can be written as a singular integral operator:
\beq\label{fractional}
\Delta_x^{\alpha/2}u (x) = c_\alpha \mbox{P.V.} \int_\bbr \frac{u(y)-u(x)}{|y-x|^{1+\alpha}} dy.
\eeq
The equation \eqref{main_1} is sometimes called the fractal Burgers equation. It has been extensively used as a toy model for the study of the fractal (anomalous) diffusion for a variety of physical phenomena where shock creation is an important ingredient. This includes the growth of molecular interfaces, traffic jams and the mass distribution for the
large scale structure of the universe 
(see for example, Biler et al. \cite{BFW} for a discussion of this model).

For $0<\alpha \leq 1,$ the well-posedness theory of \eqref{main_1} has been established in Alibaud \cite{Alibaud} and in Kiselev-Nazarov-Shterenberg \cite{kiselev} for a different class of initial data and with further analysis 
about finite time blowup for $\alpha < 1$ and analyticity for $\alpha \geq 1$ (see Chan-Czubak \cite{Chan} for $\alpha=1$). In the case of $1<\alpha < 2$, which is the focus of our work, prior to \cite{kiselev} was the work of Droniou-Gallouet-Vovelle \cite{DGV}, where the authors used a semi-group approach to obtain existence, uniqueness, smoothness and boundedness of solutions to \eqref{main_1} as well as their derivatives. Concerning time-asymptotic stability to rarefaction waves, we refer to Alibaud-Imbert-Karch \cite{AIK} and Karch-Miao-Xu \cite{KMX}.\\

In this article, we study the vanishing viscosity limit ($\eps\to0$) of the scaled equation
\begin{align}
\begin{aligned}\label{main_B}
 \left\{ \begin{array}{ll}
        \partial_t \ueps + \partial_x \left(\frac{\ueps^2}{2}\right) = \eps \Delta_x^{\alpha/2}\ueps , \quad t>0,~x\in \R,\\
         \ueps(x,0) = u_0(x), \end{array} \right.
\end{aligned}
\end{align}
in the case of $1<\alpha < 2$. \\
Note that for a solution $u$ to the equation \eqref{main_1}, $\ueps(x,t):=u(x/\eps^\beta, t/\eps^\beta)$ solves the scaled equation \eqref{main_B}, where 
\[
\beta=\frac{1}{\alpha-1}, 
\hspace{.1in} \alpha \in (1,2).
\] 
We aim to quantify the vanishing viscosity limit ($\eps\to0$) of \eqref{main_B} with a general initial datum towards entropy shock waves of the (inviscid) Burgers equation:
\beq\label{inv-main}
\partial_t u + \partial_x  \left(\frac{u^2}{2}\right) = 0.
\eeq
We are particularly interested in the case where the initial datum carries too much entropy for the structure of the layer to be preserved in the inviscid limit.\\

It is well known that for any constants $u_-$ and $u_+$ with $u_->u_+$, the equation \eqref{inv-main} admits the entropy shock wave $S_0(x-\sigma t)$ connecting the two end states $u_\pm$ as follows (for example, see \cite{Serre_book}) :
\begin{align}
\begin{aligned}\label{inv-shock}
S_0(x-\sigma t) =  \left\{ \begin{array}{ll}
       u_- \quad &\mbox{if } x<\sigma t\\
        u_+ \quad &\mbox{if } x>\sigma t, \end{array} \right.
\end{aligned}
\end{align}
where the velocity $\sigma$ is determined by the Rankine-Hugoniot condition:
\[
\sigma = \frac{A(u_-)-A(u_+)}{u_- -u_+},\quad\mbox{where~} A(u):=\frac{u^2}{2}.
\]  
Note that the condition $u_->u_+$ ensures that the shock wave \eqref{inv-shock} is an entropy solution to \eqref{inv-main}.\\
On the other hand, we refer to Chmaj \cite{Chmaj} for the existence of shock layer to the fractal Burgers equation \eqref{main_1} in the case of $1<\alpha < 2$. That is, the following was proved: 
for any $u_->u_+$, there exists a travelling wave $S_1(x-\sigma t)$ as a smooth solution to
\begin{align}
\begin{aligned}\label{shock_eq}
 \left\{ \begin{array}{ll}
        -\sigma S_1' +  \left(\frac{S_1^2}{2}\right)' = \Delta_{\xi}^{\alpha/2} S_1,\\
         \lim_{\xi\to\pm\infty} S_1(\xi) = u_\pm. \end{array} \right.
\end{aligned}
\end{align}
However, the rate of convergence of the shock layer to the two end states $u_\pm$ is not known. \\

We now present our main result.
\begin{theorem}\label{thm}
Assume $1<\alpha<2$ in the equation \eqref{main_B}.
For any constants $u_-$ and $u_+$ with $u_->u_+$, let $u_0$ be the initial datum such that
\[
u_0\in L^\infty(\bbr),\qquad u_0 -S_0 \in L^2(\bbr)\qquad\mbox{and}\qquad \left(\frac{d}{dx} u_0\right)_+ \in L^2(\bbr),
\]
where $\left(\frac{d}{dx} u_0\right)_+$ denotes the positive part of $\frac{d}{dx} u_0$.\\
For any $T>0$, there exists a constant $C(T)>0$ such that  
the following holds:\\
For any solution $\ueps$ to \eqref{main_B}, there exists a  Lipschitz continuous shift $t\in[0,T]\mapsto X_\eps(t)$ with $X_\eps(0)=0$ such that 
for all $t\le T$,
\beq\label{result1}
 \|\ueps(\cdot+X_\eps(t),t)-S_0(\cdot-\sigma t)\|_{L^2(\R)} \le \|u_0 -S_0\|_{L^2(\R)}   + C(T) \psi(\eps) .
\eeq
Here,
\[
\psi(\eps):=\inf_{\delta\ge 4} \left(\sqrt{\delta\eps^\beta +  \sqrt{\delta\eps^\beta} +  \left( S_1 \left(\sqrt\delta\right) - u_+ \right) + \left( u_- - S_1 \left(-\sqrt\delta\right)  \right)  + \left( \frac{1}{\delta}\right)^{\alpha-1}  } \right) + \eps^{\beta/2} ,
\]
where $S_1$ denotes the viscous shock satisfying \eqref{shock_eq}.
\end{theorem}

\vspace{0.2cm}
\begin{remark}
Note that the shift $X_\eps$ depends both on $\eps$ and the initial value $u_0$. The shift cannot be reduced to the actual velocity of the shock, since at the limit $\eps$ goes to zero, the contraction in $L^2$ without extra shift is false (see Leger \cite{Leger}). 
The shift $X_\eps$ will be constructed as   a solution to the ODE \eqref{ODE}.  In the following sentences of  \eqref{ODE}, we will justify the existence and uniqueness of the Lipschitz continuous solution $X_\eps$.
In what follows, we drop the $\eps$-dependence of the shift $X_\eps$ for simplicity.
\end{remark}

\begin{remark}
Since $\lim_{\xi\to\pm\infty} S_1(\xi) = u_\pm$ and $1<\alpha<2$, note that (for example, by choosing $\delta=\eps^{-\beta/2}$)
\[
\psi(\eps)  \to 0\quad \mbox{as } \eps\to 0.
\]
Therefore, Theorem \ref{thm} provides an explicit rate of convergence for the inviscid limit to the shock.\\
If the shock layer $S_1$ approaches the end states $u_\pm$ exponentially fast as in the case of the classical Laplacian, i.e., $\alpha=2$ (for example, see \cite{Kang-V-1}), then there exist constants $\eps_0$ and $C>0$ such that
\beq\label{thm_rate}
\psi(\eps) \le C \eps^{\frac{1}{2(2\alpha-1)}} ,\quad \forall \eps<\eps_0.
\eeq
Indeed, by choosing
\[
\delta = \eps^{-\frac{1}{(\alpha-1)(2\alpha-1)}} 
\]
there exists $\eps_0>0$ such that for all $0<\eps<\eps_0$,
\begin{align}
\begin{aligned}\label{main}
& \delta\eps^\beta \ll \sqrt{\delta \eps^\beta} = \eps^{\frac{1}{2\alpha-1}}, \\
& \left| S_1 \left(\sqrt{\pm\delta}\right) - u_\pm \right| \le C e^{- c_\pm \sqrt\delta}  \ll \left( \frac{1}{\delta}\right)^{\alpha-1} =\eps^{\frac{1}{2\alpha-1}}. 
\end{aligned}
\end{align}
\end{remark}

\begin{remark}
From a special layer study, one can see that the optimal rate of convergence  is $\eps^{\beta/2}$. Indeed, if we consider the shock layer $S_1(x-\sigma t)$ in \eqref{shock_eq},
then $S_1((x-\sigma t)/\eps^\beta)$ is a shock layer of \eqref{main_B} as a travelling wave solution of \eqref{main_B} with initial datum $S_1(x/\eps^\beta)$. 
In this case, the rate of convergence is of order $\eps^{\beta/2}$, since
\beq\label{shock-conv}
\sqrt{\int_\R \left|S_1\left(\frac{x-\sigma t}{\eps^\beta}\right) - S_0 (x-\sigma t) \right|^2 \dx }= \sqrt{\eps^\beta \int_\R \left|S_1\left(x\right) - S_0 (x) \right|^2 \dx }=C\eps^{\beta/2} .
\eeq
Therefore, if the shock layer $S_1$ approaches the end states exponentially fast, the rate of convergence $\eps^{\frac{1}{2(2\alpha-1)}}$  in \eqref{thm_rate} is slightly worse than the optimal rate $\eps^{\beta/2}$ above, because
\[
\eps^{\beta/2} = \eps^{\frac{1}{2(\alpha-1)}} < \eps^{\frac{1}{2(2\alpha-1)}}.
\]
Note that such a layer study is the special case of small initial perturbations such as
\[
 \|u_0 -S_0\|_{L^2(\R)}  = \mathcal{O}(\eps^{\beta/2}) .
\]
In the  case where $u_0$ is the same initial data as the one of \eqref{inv-main}, i.e, no initial perturbation, we refer to the result of Droniou \cite{Droniou} on the convergence of solution to \eqref{main_B} towards entropy solution to \eqref{inv-main}. \\
However, those studies collapse in the case of large initial perturbation as
\[
 \|u_0 -S_0\|_{L^2(\R)} \gg \eps^{\beta/2}.
\] 
In this situation, there is too much entropy for the asymptotic limit of the layer structure to be true. 
So, the physical layer may be destroyed. 
Therefore, Theorem \ref{thm} is the first result on the inviscid limit to the entropy shock even for large initial perturbation, although the rate of convergence is not optimal.
\end{remark}

\vspace{1cm}

\section{Proof of Theorem \ref{thm}}
We prove Theorem \ref{thm} for the case of a conservation law with a strictly flux: given a strictly convex flux $A$, consider
\begin{align}
\begin{aligned}\label{main}
 \left\{ \begin{array}{ll}
        \partial_t \ueps + \partial_x A(\ueps) = \eps \Delta_x^{\alpha/2}\ueps , \quad t>0,~x\in \R.\\
         \ueps(x,0) = u_0(x), \end{array} \right.
\end{aligned}
\end{align}
Although the existence issue of the shock layer of \eqref{main} is still open for general convex fluxes, we provide the proof of Theorem \ref{thm} in the general setting.\\
We also mention that Droniou-Gallouet-Vovelle \cite{DGV} proved the global existence and uniqueness of smooth solutions to \eqref{main} with the $L^\infty$-bounded initial data in the case of $1<\alpha<2$, and
\[
\|\ueps\|_{L^\infty(\R)}\le \|u_0\|_{L^\infty(\R)},
\]
which will be used in our proof.

\subsection{Ideas and useful lemma}
Contrary to the proof of the result \cite{CV} for the case of the (local) Laplacian operator, i.e., $\alpha=2$, 
the nonlocality of the fractional Laplacian leads us to first study on the convergence of the solution $\ueps$ towards the shock layer (of width $\eps^\beta$) of \eqref{main}.
Once we prove it, the desired result \eqref{result1} would be obtained by using the obvious convergence from the shock layer to the inviscid shock as in \eqref{shock-conv}.     

Without loss of generality, we only deal with the stationary shock wave $S_0$, i.e., $\sigma=0$.
We first see from \eqref{shock_eq} that the (stationary) shock layer $S_1$ of \eqref{main_1} is a solution to
\begin{align}
\begin{aligned}\label{eq_S1}
 \left\{ \begin{array}{ll}
      (A(S_1))' = \Delta_{x}^{\alpha/2} S_1,\\
         \lim_{x\to\pm\infty} S_1(x) = u_\pm. \end{array} \right.
\end{aligned}
\end{align}
Then, $\Seps(x):=S_1(x/\eps^\beta)$ is the associated shock layer of \eqref{main} as a solution to
\begin{align}
\begin{aligned}\label{eq_Seps}
 \left\{ \begin{array}{ll}
      (A(\Seps))' =\eps \Delta_{x}^{\alpha/2} \Seps,\\
         \lim_{x\to\pm\infty} \Seps(x) = u_\pm. \end{array} \right.
\end{aligned}
\end{align}
In our analysis, we will use the monotonicity property of the shock layer, which is proved in the following lemma.
\begin{lemma}\label{lem:mono}
If $S_1(x)$ is a smooth shock layer of \eqref{eq_S1}, then $S_1'(x) \leq 0$ for all $x.$
\end{lemma}
\begin{proof}
First, we take the derivative of both sides of (\ref{eq_S1}) to get
\beq\label{dif}
		A''(S_1) |S_1'|^2 + A'(S_1)S_1'' = \Delta^{\alpha/2}S_1'.
\eeq
Multiplying both sides of (\ref{dif}) by $(S_1')_+$ and using 
$S_1'' \mathbf{1}_{S_1'>0}=((S_1')_+)'$, we have
\[
A''(S_1) ((S_1')_+)^3 +A'(S_1) \left( \frac{((S_1')_+)^2}{2} \right)' = (\Delta^{\alpha/2} S_1')(S_1')_+ .
\]
We integrate both sides over $\R$ to get
\[
\int  A''(S_1) ((S_1')_+)^3 \,\dx 
+ \int A'(S_1) \left( \frac{((S_1')_+)^2}{2} \right)'\dx = \int (S_1')_+ \Delta^{\alpha/2} (S_1')  \,\dx.
\]
Note that using \eqref{fractional} and anti-symmetry, and the fact that $f=f_+ - f_-$ where $f_+$ and $f_-$ denote respectively the positive and negative parts of a function $f$, we have
\begin{align}
\begin{aligned}\label{genpa}
 &\int (S_1')_+ \Delta^{\alpha/2} (S_1')  \,\dx
  =-\int_{\R^2}(S_1')_+(x)\frac{S_1'(x)-S_1'(y)}{|x-y|^{1+\alpha}}\,dx\,dy\\
 &\quad=-\frac{1}{2}\int_{\R^2}\frac{((S_1')_+(x)-(S_1')_+(y))(S_1'(x)-S_1'(y))}{|x-y|^{1+\alpha}}\,dx\,dy\\
& \quad=-\frac{1}{2}\int_{\R^2}\frac{((S_1')_+(x)-(S_1')_+(y))^2}{|x-y|^{1+\alpha}}\,dx\,dy \\
&\qquad +\frac{1}{2}\int_{\R^2}\frac{((S_1')_+(x)-(S_1')_+(y))((S_1')_-(x)-(S_1')_-(y))}{|x-y|^{1+\alpha}}\,dx\,dy\\
& \quad=-\frac{1}{2}\int_{\R^2}\frac{((S_1')_+(x)-(S_1')_+(y))^2}{|x-y|^{1+\alpha}}\,dx\,dy -\int_{\R^2}\frac{(S_1')_+(x)(S_1')_-(y)}{|x-y|^{1+\alpha}}\,dx\,dy \le 0.
\end{aligned}
\end{align}
Moreover, since
\[
\int  A''(S_1) ((S_1')_+)^3 \,\dx 
+ \int A'(S_1) \left( \frac{((S_1')_+)^2}{2} \right)'\dx = \frac{1}{2}\int A''(S_1)((S_1')_+)^3\dx,
\]
we have
\[
\frac{1}{2}\int  A''(S_1) ((S_1')_+)^3 \,\dx \le 0.
\]
Therefore, using the strict convexity of the flux $A$, we have
\[
(S_1')_+ (x)=0,\qquad \mbox{for a.e. }~ x\in\R,
\]
which completes the proof.
\end{proof}

The following lemma will be used in the proof of Proposition \ref{prop:h1}.

\begin{lemma}\label{lem:L2}
If $u_\eps(x,t)$ is a solution of (\ref{main}) with $ \|(\partial_xu_0)_+\|_{L^2(\R)}<\infty$, then
\begin{equation*}
	\|(\partial_x\ueps(\cdot,t))_+\|_{L^2(\R)}
	\leq  \|(\partial_xu_0)_+\|_{L^2(\R)}.
\end{equation*}
\end{lemma}
\begin{proof}

Let $v: = \partial_x u_\eps$ and $v_+: = v\mathbf{1}_{v\geq 0}.$
Following the proof of \cite[Lemma 3.2]{CV}, we differentiate (\ref{main}) with respect to $x,$ multiply by $(\partial_x(u_\eps))_+$ and integrate in $x$ to get 
\[
\int (\partial_t v) v_+\dx
		+ \int \partial_x\left[A'(u_\eps) v\right]v_+\dx
		= \eps \int (\Delta_x^{\alpha/2}v) v_+\dx.
\]
Then, using the same estimates as in \eqref{genpa}, we have
\[
\int (\partial_t v) v_+\dx
		+ \int \partial_x\left[A'(u_\eps) v\right]v_+\dx
		= \eps \int (\Delta_x^{\alpha/2}v) v_+\dx\\
		\leq 0 .
\]
Moreover, since
\[
\int (\partial_t v) v_+\dx
		 = \int \partial_t (v_+) v_+\dx
		= \frac{1}{2}\frac{\mathrm{d}}{\dt}\int (v_+)^2 \dx \label{term1},
\]
and 
\begin{align*}
\begin{aligned}
\int \partial_x\left[A'(u_\eps) v\right]v_+\dx
		& = \int A''(u_\eps) |v|^2 v_+\dx
		+ \int A'(x)( \partial_x v)v_+\dx \\
		& = \int A''(u_\eps) (v_+)^3\dx
		+\int A'(x)\partial_x \Big( \frac{(v_+)^2}{2} \Big)\dx \\
		& = \frac{1}{2}\int A''(u_\eps) (v_+)^3\dx,
\end{aligned}
\end{align*}	
we find that
\[
		\frac{\mathrm{d}}{\dt}\|(\partial_x\ueps(\cdot,t))_+\|_{L^2(\R)}^2
		= \frac{\mathrm{d}}{\dt} \int  (v_+)^2\dx
		\leq 	-  \int A''(u_\eps) (v_+)^3 \dx
		\leq 0,
\]
which completes the proof. 
\end{proof}

\subsection{Evolution of the relative entropy}

Let $\vphi$ be a smooth nondecreasing function such that
\begin{align}
\begin{aligned}\label{rough-phi}
\vphi(x) = \left\{ \begin{array}{ll}
        0  &\mbox{if }~ x\le 0,\\
         1 &\mbox{if }~ x\ge 1. \end{array} \right.
\end{aligned}
\end{align}
To localize the layer, we consider a parametrized function $\vphid$, $\delta >0$, defined by 
\[
\vphid(x):=\vphi \left(\frac{x}{\delta}\right).
\]
The parameter $\delta$ will be determined as a function of $\eps$ at the end of the proof.

For fixed $\eps, \delta>0$ and $X\in C^1([0,T])$, we will consider the evolution of 
\beq\label{defH}
\mathcal H(t):=\int_\R \vphid^2\left(\frac{|x|}{\eps^\beta}\right) \frac{ |\ueps(x+X(t),t)-\Seps(x)|^2 }{2} \, dx.
\eeq
Although the above functional is based on the $L^2$-norm, we take advantage of the relative entropy method to get the convergence of $\mathcal H'(t)$ as in \cite{CV}.

The relative entropy method was introduced in the studies by Dafermos \cite{Dafermos1} and Diperna \cite{DiPerna} of $L^2$-stability and uniqueness of Lipschitz solutions to hyperbolic conservation laws endowed with a convex entropy. Recently, this method was extensively used in studying the contraction and inviscid limit for large initial perturbations of viscous (or inviscid) shock waves (see \cite{CV,Kang19,Kang18_KRM,Kang-V-NS17,KV-unique19,KVARMA,Kang-V-1,KVW,Leger,Serre-Vasseur,SV_16,SV_16dcds,Vasseur_Book,Vasseur-2013,VW}).

To use the relative entropy method, in particular we consider the quadratic entropy 
\beq\label{def-ent}
\eta(u):=\frac{u^2}{2},
\eeq
where we  note from the theory of conservation laws that any function is an entropy of the scalar conservation law \eqref{main}. \\
In the general theory, for a strictly convex entropy $\eta$, we define the associated relative entropy function by
\beq\label{def-rel}
\eta(u|v):=\eta(u)-\eta(v) -\eta^{\prime}(v) (u-v).
\eeq
Likewise, we define the relative functional of the strictly convex flux $A$ by
\beq\label{def-A}
A(u|v):=A(u)-A(v) -A'(v) (u-v).
\eeq
Let $F(\cdot,\cdot)$ be the flux of the relative entropy defined by
\beq\label{def-F}
F(u,v) := G(u)-G(v) -\eta^{\prime} (v) (A(u)-A(v)),
\eeq
where $G$ is the entropy flux of $\eta$, i.e., $G^{\prime} = \eta^{\prime} A^{\prime}$.\\

Since, for the quadratic entropy \eqref{def-ent}, the associated relative entropy is
\beq\label{reluv}
\eta(u|v)=\frac{|u-v|^2}{2},
\eeq
the function $\mathcal H(t)$ in \eqref{defH} can be rewritten as
\[
\mathcal H(t):=\int_\R \vphid^2\left(\frac{|x|}{\eps^\beta}\right) \eta( \ueps(x+X(t),t)|\Seps(x)) \,dx .
\]

For simplification of our presentation, we use a change of variable as follows:
\beq\label{change-v}
\ve(x,t):=\ueps( x+X(t), t).
\eeq
Then, it follows from \eqref{main} that $\ve$ satisfies
\begin{align}
\begin{aligned} \label{v-eq}
\partial_t \ve -\dot{X}(t) \partial_x \ve + \partial_x A(\ve) =  \eps \Delta_x^{\alpha/2} \ve .
\end{aligned}
\end{align}

We now present the following lemma.

\begin{lemma}\label{lem:base}
The function $\mathcal H(t)$ defined by \eqref{defH} satisfies
\begin{align}
\begin{aligned} \label{est-H}
		\mathcal H'(t)
		& = \int_{-\infty}^\infty \vphid^2\left(\frac{|x|}{\eps^\beta}\right)\partial_x \left(\eta(\ve(x,t)|\Seps(x))\dot {X}(t) - F(\ve(x,t),\Seps(x))\right)\dx \\
		& \quad+\int_{-\infty}^\infty \vphid^2\left(\frac{|x|}{\eps^\beta}\right)(\ve(x,t) - \Seps(x))\Seps'(x)\left(\dot {X}(t) - \frac{A(\ve(x,t)|\Seps(x))}{\ve(x,t)-\Seps(x)}\right)\dx \\
		& \quad+\eps\int_{-\infty}^\infty \vphid^2\left(\frac{|x|}{\eps^\beta}\right)(\ve(x,t) - \Seps(x))\Delta^{\alpha/2}(\ve(x,t) - \Seps(x))\dx \\
		&=: H_1 +  H_2 + P.
\end{aligned}
\end{align}
\end{lemma}

\begin{proof}
First of all, since 
\[
\mathcal H(t) =  \int_\R \vphid^2\left(\frac{|x|}{\eps^\beta}\right)\eta\left(\ve(x,t)|\Seps\left(x\right)\right)\,dx,
\]
we have
\[
\mathcal H'(t) =   \int_\R \vphid^2\left(\frac{|x|}{\eps^\beta}\right)\partial_t\Big(\eta\left(\ve(x,t)|\Seps\left(x\right)\right) \Big)\,dx.
\]
Note from the definition \eqref{def-rel} that
\begin{align*}
\begin{aligned} 
\partial_t \eta(\ve(x,t)|\Seps(x))  = (\eta'(\ve) - \eta'(\Seps)) \partial_t \ve(x,t) - \eta''(\Seps) (\ve-\Seps) \underbrace{\partial_t \Seps(x)}_{=0}.
\end{aligned}
\end{align*} 
To get a nice quadratic structure from the above right-hand side, we use \eqref{eq_S1} and \eqref{v-eq} so that
\begin{align*}
\begin{aligned}
\partial_t \eta(\ve(x,t)|\Seps(x))  &= (\eta'(\ve) - \eta'(\Seps))  \left( \dot{X}(t) \partial_x \ve - \partial_x A(\ve) + \eps \Delta_x^{\alpha/2} \ve \right) \\
&\quad - \eta''(\Seps) (\ve-\Seps) \left( - (A(\Seps))' + \Delta_{x}^{\alpha/2} \Seps \right) \\
& =\dot{X}(t)(\eta'(\ve) - \eta'(\Seps)) \partial_x \ve \\
&\quad - (\eta'(\ve) - \eta'(\Seps)) \partial_x A(\ve) + \eta''(\Seps) (\ve-\Seps)(A(\Seps))' \\
&\quad + \eps \left(  (\eta'(\ve) - \eta'(\Seps)) \Delta_x^{\alpha/2} \ve - \eta''(\Seps) (\ve-\Seps)\Delta_{x}^{\alpha/2} \Seps \right).
\end{aligned}
\end{align*}
Since a straightforward computation together with the definitions \eqref{def-A} and \eqref{def-F} yields the identity
\[
-\partial_x F(\ve,\Seps) -\eta'' (\Seps) \Seps' A(\ve|\Seps) 
= - (\eta'(\ve) - \eta'(\Seps)) \partial_x A(\ve) + \eta''(\Seps) (\ve-\Seps)(A(\Seps))' 
\]
(which also appears in the proof of \cite[Lemma 2.1]{Kang-V-1}), we have
\begin{align*}
\begin{aligned}
\partial_t \eta(\ve(x,t)|\Seps(x)) 
& =\dot{X}(t)(\eta'(\ve) - \eta'(\Seps)) \partial_x \ve \\
&\quad -\partial_x F(\ve,\Seps) -\eta'' (\Seps) \Seps' A(\ve|\Seps) \\
&\quad + \eps \left(  (\eta'(\ve) - \eta'(\Seps)) \Delta_x^{\alpha/2} \ve - \eta''(\Seps) (\ve-\Seps)\Delta_{x}^{\alpha/2} \Seps \right).
\end{aligned}
\end{align*}
We now use the quadratic entropy \eqref{def-ent} to obtain
\begin{align*}
\begin{aligned}
\partial_t \eta(\ve(x,t)|\Seps(x)) 
& =\dot{X}(t)(\ve - \Seps) \partial_x \ve -\partial_x F(\ve,\Seps) - \Seps' A(\ve|\Seps) \\
&\quad + \eps  (\ve - \Seps)   \Delta_x^{\alpha/2} \left(  \ve - \Seps \right).
\end{aligned}
\end{align*}
Therefore, using \eqref{reluv} and
\[
(\ve - \Seps) \partial_x \ve =  \partial_x \left(\frac{|\ve - \Seps|^2}{2}\right) +  (\ve - \Seps)\partial_x\Seps,
\]
we complete the proof.
\end{proof}

\begin{remark}
Contrary to \cite[Lemma 2.1]{CV}, we have a new hyperbolic part ${H}_2$ in \eqref{est-H}, because we are considering the viscous layer. 
\end{remark}

\subsection{Estimate on the first hyperbolic part ${H}_1$}
Here, we estimate the first part ${H}_1$ of $\mathcal{H}'(t)$ in (\ref{est-H}) by following the strategy in \cite[Section 3]{CV}. For this, we consider the normalized relative entropy flux $f(\cdot,\cdot),$ given by
\beq\label{def-relf}
f(u,v) :=\frac{F(u,v)}{\eta(u|v)}.
\eeq
With such an $f,$ we have the following properties.

\begin{lemma}\label{lem:f}
For any $L>0$, there exists a constant $\Lambda>0$ such that for any ${u},v$ with $|{u}|, |v| \leq L,$
\begin{align*}
		& 0\leq  \partial_{u}f({u},v) \leq \Lambda ,\\
		& \frac{1}{\Lambda} \leq \partial_vf({u},v).
	\end{align*}
\end{lemma}

For the proof of the above lemma, we refer to \cite{Leger}.\\

We now define the shift $X$ as a solution to the ODE
\begin{align}
\begin{aligned}\label{ODE}
 \left\{ \begin{array}{ll}
         \dot {X}(t) = f({v_\eps}(0,t), \Seps(0)),\\
         X(0)=0, \end{array} \right.
\end{aligned}
\end{align}
where recall from \eqref{change-v} that $\ve(0,t)=\ueps(X(t),t)$.
As mentioned before, since, for any $\eps>0$, the equation \eqref{main} with $L^\infty$ initial datum admits a unique smooth solution, the Cauchy-Lipschitz theorem together with Lemma \ref{lem:f} implies
the existence and uniqueness of the solution $X$ to the ODE \eqref{ODE}.\\

We now present a bound on $H_1$ in \eqref{est-H}.
In what follows, $C$ denotes a positive constant which may change from line to line, but which is independent on $\eps$.

\begin{proposition}\label{prop:h1}
Let $\vphi$ be a smooth nondecreasing function satisfying \eqref{rough-phi}. Under the same hypotheses as in Theorem \ref{thm}, there exists a positive constant $C=C(\|u_0\|_{L^\infty(\R)}, u_\pm)$ such that for any $\eps, \delta>0$,
	\[
		H_1= \int_{-\infty}^\infty \vphid^2\left(\frac{|x|}{\eps^\beta}\right)\partial_x \left(\eta(\ve(x,t)|\Seps(x))\dot {X}(t) - F(\ve(x,t),\Seps(x))\right)\dx	\leq C \sqrt{\delta\eps^\beta}.
	\]	
\end{proposition}

\begin{proof}
First of all, we separate $H_1$ into two parts:
\begin{align*}
\begin{aligned}
H_1(t)&=\int_{-\infty}^0 \left(\vphid\left(-\frac{x}{\eps^\beta}\right)\right)^2 \partial_x\left(\eta({v_\eps}(x,t)|\Seps(x))\dot {X}(t) - F({v_\eps}(x,t)), \Seps(x)\right)\dx  \\
&\quad +\int_0^\infty \left(\vphid\left(\frac{x}{\eps^\beta}\right)\right)^2 \partial_x\left(\eta({v_\eps}(x,t)|\Seps(x))\dot {X}(t) - F({v_\eps}(x,t)), \Seps(x)\right)\dx \\
&=: H_1^L(t) + H_1^R(t).
\end{aligned}
\end{align*}
For $H_1^L$, by an integration by parts together with \eqref{def-relf}, \eqref{reluv} and \eqref{ODE}, we have
\begin{align*}
\begin{aligned}
&H_1^L(t) \\
&= - \int_{-\infty}^0 \partial_x\left[\left(\vphi\left(-\frac{x}{\delta\eps^\beta}\right)\right)^2\right]\left(\eta({v_\eps}(x,t)|\Seps(x))\dot {X}(t) - \eta({v_\eps}(x,t)|\Seps(x))f({v_\eps}(x,t), \Seps(x))\right)\dx \nonumber \\
		&= -\int_{-\infty}^0 \left[-\frac{2}{\delta\eps^\beta}\vphi\left(-\frac{x}{\delta\eps^\beta}\right)\vphi'\left(-\frac{x}{\delta\eps^\beta}\right)\right]\frac{|v_\eps(x,t)-\Seps(x)|^2}{2}\left(\dot {X}(t) - f({v_\eps}(x,t), \Seps(x))\right)\dx \nonumber \\
		&= \frac{1}{\delta\eps^\beta}\int_{-\delta\eps^\beta}^0 \vphi\left(-\frac{x}{\delta\eps^\beta}\right)\vphi'\left(-\frac{x}{\delta\eps^\beta}\right) |v_\eps(x,t)-\Seps(x)|^2  h(x,t) \dx \nonumber,
\end{aligned}
\end{align*}
where 
\[
h(x,t):= f({v_\eps}(0,t), \Seps(0)) - f({v_\eps}(x,t), \Seps(x)).
\]
Notice that since $\|\ueps\|_{L^\infty(\R)}\le \|u_0\|_{L^\infty(\R)}$ by the maximum principle, and $|S_\eps|\leq \max\{{|u_-|}, {|u_+|}\},$ there is a constant 
$C=C(\|u_0\|_{L^\infty(\R)}, u_\pm)$ such that
\beq\label{bdd-vs}
\|v_\eps-\Seps\|_{L^\infty([0,T]\times\R)} \le C.
\eeq
To control $h(x,t)$, we first separate it into two parts:
\begin{align*}
\begin{aligned}
h(x,t)= \underbrace{\left(f({v_\eps}(0,t),\Seps(0)) - f({v_\eps}(x,t),\Seps(0))\right)}_{=:h_1}
		+\underbrace{ \left(f({v_\eps}(x,t),\Seps(0)) - f({v_\eps}(x,t), \Seps(x))\right)}_{=:h_2} .
\end{aligned}
\end{align*}
To estimate $h_1$, using Lemma \ref{lem:L2}, we observe  that for any $x<0$,
\[
v_\eps(0,t)- v_\eps(x,t) = \int_{x+X(t)}^{X(t)} \partial_y u_\eps(y,t) dy \le  \int_{x+X(t)}^{X(t)} (\partial_y u_\eps)_+ (y,t) dy \le  \|\left(\partial_xu_0\right)_+\|_{L^2(\R)} \sqrt{|x|}.
\] 
Then, since $f$ is increasing with respect to the first variable by Lemma \ref{lem:f}, we have
\begin{align*}
\begin{aligned}
h_1&\le f\Big( v_\eps(x,t) + \|\left(\partial_xu_0\right)_+\|_{L^2(\R)} \sqrt{|x|},\Seps(0)\Big) - f({v_\eps}(x,t),\Seps(0)) \\
		&\leq \Lambda\sqrt{|x|} \|\left(\partial_xu_0\right)_+\|_{L^2(\R)}.
\end{aligned}
\end{align*}
Using Lemma \ref{lem:f} and Lemma \ref{lem:mono}, we have
\[
h_2 \le \frac{1}{\Lambda} (\Seps(0)-\Seps(x))\le 0\quad \forall x\le 0.
\]
Therefore, we have
\begin{align*}
\begin{aligned}
H_1^L(t) &\le \frac{C}{\delta\eps^\beta}\int_{-\delta\eps^\beta}^0 \vphi\left(-\frac{x}{\delta\eps^\beta}\right)\vphi'\left(-\frac{x}{\delta\eps^\beta}\right) \sqrt{|x|}\, \dx \nonumber &\le C\sqrt{\delta\eps^\beta} \int_{-1}^0 \vphi(-x) \vphi'(-x) \sqrt{|x|} dx\\
&\le C\sqrt{\delta\eps^\beta},
\end{aligned}
\end{align*}
where the last inequality is obtained by the definition of $\vphi$ as 
\[
\int_{-1}^0 \vphi(-x) \vphi'(-x) \sqrt{|x|} dx \le \int_{-1}^0 \vphi'(-x) dx = \vphi(1)-\vphi(0)=1.
\]
Likewise, using the same method as above, we have
\[
H_1^R(t) \le C\sqrt{\delta\eps^\beta}.
\]
Hence, we complete the proof.
\end{proof}

\subsection{Estimate on the second hyperbolic part ${H}_2$}

Here we find a bound for convergence of the second part $H_2$ in \eqref{est-H}. For this,  we consider a specific choice of the monotone function $\vphi$ satisfying \eqref{rough-phi}, defined by
\begin{align}
\begin{aligned}\label{special-phi}
\vphi(x) = \left\{ \begin{array}{ll}
        0  &\mbox{if }~ x\le 0,\\
         2x^2 &\mbox{if }~ 0 \le x\le 1/2,\\
         1-2(x-1)^2 &\mbox{if }~ 1/2 \le x\le 1,\\
         1 &\mbox{if }~ x\ge 1. \end{array} \right.
\end{aligned}
\end{align}

\begin{proposition}\label{prop:h2}
Let $\vphi$ be the function defined by \eqref{special-phi}.
Under the same hypotheses as in Theorem \ref{thm}, there exists a positive constant $C$ such that for any $\delta\ge 4$,
\begin{align*}
\begin{aligned}
	H_2&=\int_{-\infty}^\infty \vphid^2\left(\frac{|x|}{\eps^\beta}\right)(\ve(x,t) - \Seps(x))\Seps'(x)\left(\dot {X}(t) - \frac{A(\ve(x,t)|\Seps(x))}{\ve(x,t)-\Seps(x)}\right)\dx	\\
	&\leq  C \left(  \left( \frac{1}{\delta}\right)^{3/2} +  \left( S_1 \left(\sqrt\delta\right) - u_+ \right) + \left( u_- - S_1 \left(-\sqrt\delta\right)  \right)  \right),
\end{aligned}
\end{align*}
\end{proposition}

\begin{proof}
Since $\ve$ and $\Seps$ are bounded as mentioned in \eqref{bdd-vs}, using the definition \eqref{def-A} of $A(\cdot|\cdot)$, and \eqref{ODE} with Lemma \ref{lem:f}, we observe that there exists a positive constant $C=C(\|u_0\|_{L^\infty(\R)}, u_\pm)$ such that for all $x\in\bbr$ and $t\le T$,
\[
A(\ve|\Seps)=(\ve-\Seps)^2\int_0^1\int_0^1 A''(\Seps + st(\ve-\Seps)) t dsdt \le \|A''\|_{L^\infty(-C,C)} |\ve-\Seps|^2,
\] 
and
\[
|\dot {X}(t)| \le C.
\]
Therefore, using \eqref{bdd-vs}, we have
\[
H_2 \le C \int_{-\infty}^\infty \vphid^2\left(\frac{|x|}{\eps^\beta}\right) |\Seps'(x)|\, \dx.
\]
Note that, since $\Seps'(x) = \eps^{-\beta}S_1'(x\eps^{-\beta})$, 
\[
H_2 \le C \int_{-\infty}^\infty \vphid^2\left(|x|\right) |S_1'(x)| \, \dx.
\]
We now separate the right hand side into two parts:
\[
\int_{-\infty}^\infty \vphid^2\left(|x|\right) |S_1'(x)| \, \dx =\int_{|x|\le\sqrt{\delta}} \vphid^2\left(|x|\right) |S_1'(x)| \, \dx + \int_{|x|\ge\sqrt{\delta}} \vphid^2\left(|x|\right) |S_1'(x)| \, \dx .
\]
For any $\delta\ge 4$, since
\[
|x|\le \sqrt\delta \quad \Rightarrow \quad \frac{|x|}{\delta} \le \frac{1}{\sqrt\delta} \le \frac{1}{2},
\]
we use \eqref{special-phi} to get
\begin{align*}
\begin{aligned}
\int_{|x|\le\sqrt{\delta}} \vphid^2\left(|x|\right) |S_1'(x)| \, \dx &\le \|S_1'\|_{L^\infty(\bbr)} \int_{|x|\le\sqrt{\delta}}  \vphi^2\left(\frac{|x|}{\delta}\right)\,\dx  \\
&= \|S_1'\|_{L^\infty(\bbr)}  \frac{4}{\delta^4}  \int_{|x|\le\sqrt{\delta}} |x|^4 \,\dx \le C \left( \frac{1}{\delta}\right)^{3/2}.
\end{aligned}
\end{align*}
Using Lemma \ref{lem:mono}, we have
\begin{align*}
\begin{aligned}
 \int_{|x|\ge\sqrt{\delta}} \vphid^2\left(|x|\right) |S_1'(x)| \, \dx & \le \int_{|x|\ge\sqrt{\delta}}  |S_1'(x)| \, \dx =  - \int_{|x|\ge\sqrt{\delta}}  S_1'(x) \, \dx \\
&= \left( S_1 \left(\sqrt\delta\right) - u_+ \right) + \left( u_- - S_1 \left(-\sqrt\delta\right)  \right) .
\end{aligned}
\end{align*}
Hence we complete the proof.
\end{proof}

\subsection{Estimate on the parabolic part $P$}

\begin{proposition}\label{prop:para}
Let $\vphi$ be a smooth nondecreasing function satisfying \eqref{rough-phi}. Under the same hypotheses as in Theorem \ref{thm}, there exists a positive constant $C$ such that for any $\delta> 0$,
\begin{align*}
\begin{aligned}
	P=\eps\int_\bbr \vphid^2\left(\frac{|x|}{\eps^\beta}\right)(\ve(x,t) - \Seps(x))\Delta^{\alpha/2}(\ve(x,t) - \Seps(x))\dx	\leq  C \left( \frac{1}{\delta}\right)^{\alpha-1} .
\end{aligned}
\end{align*}
\end{proposition}

\begin{proof}
For simplicity, here we set 
\[
w_\eps(x,t):=\ve(x,t) - \Seps(x).
\]
First, using \eqref{fractional} and anti-symmetry, we have
\begin{align*}
\begin{aligned}
P & = \eps \int_{\R} \vphid^2\left(\frac{|x|}{\eps^\beta}\right) {w_\eps}(x,t)\Delta^{\alpha/2}{w_\eps}(x,t) \, \dx \\
& = \eps \iint_{\R\times \R} \vphid^2\left(\frac{|x|}{\eps^\beta}\right) {w_\eps}(x,t)\frac{{w_\eps}(y,t) - {w_\eps}(x,t)}{|x-y|^{1+\alpha}}\dy\dx \\
& = -\frac{\eps}{2} \iint_{\R\times \R} \Big[\vphid^2\left(\frac{|x|}{\eps^\beta}\right) {w_\eps}(x,t) - \vphid^2\left(\frac{|y|}{\eps^\beta}\right) {w_\eps}(y,t)\Big]\frac{{w_\eps}(x,t) - {w_\eps}(y,t)}{|x-y|^{1+\alpha}}\dy\dx ,
\end{aligned}
\end{align*}
which can be rewritten into
\begin{align*}
\begin{aligned}
P & = -\frac{\eps}{2} \iint_{\R\times \R} \Big[\vphid \left(\frac{|x|}{\eps^\beta}\right) {w_\eps}(x,t) - \vphid \left(\frac{|y|}{\eps^\beta}\right) {w_\eps}(y,t)\Big] \vphid \left(\frac{|x|}{\eps^\beta}\right) \frac{{w_\eps}(x,t) - {w_\eps}(y,t)}{|x-y|^{1+\alpha}}\dy\dx \\
 &\quad -\frac{\eps}{2} \iint_{\R\times \R} \Big[\vphid \left(\frac{|x|}{\eps^\beta}\right)  - \vphid \left(\frac{|y|}{\eps^\beta}\right) \Big] \vphid \left(\frac{|y|}{\eps^\beta}\right) {w_\eps}(y,t) \frac{{w_\eps}(x,t) - {w_\eps}(y,t)}{|x-y|^{1+\alpha}}\dy\dx \\
 &=:P_1 +P_2 .
\end{aligned}
\end{align*}
Since
\begin{align*}
\begin{aligned}
P_1 & = -\frac{\eps}{2} \iint_{\R\times \R}\left[ \vphid \left(\frac{|x|}{\eps^\beta}\right){w_\eps}(x,t) -\vphid \left(\frac{|y|}{\eps^\beta}\right){w_\eps}(y,t) \right]^2 \frac{1}{|x-y|^{1+\alpha}}\dy\dx   \\
&\quad +\frac{\eps}{2} \iint_{\R\times \R}\left[ \vphid \left(\frac{|x|}{\eps^\beta}\right){w_\eps}(x,t) -\vphid \left(\frac{|y|}{\eps^\beta}\right){w_\eps}(y,t) \right]
\frac{\left[ \vphid \left(\frac{|x|}{\eps^\beta}\right) -\vphid \left(\frac{|y|}{\eps^\beta}\right) \right] {w_\eps}(y,t)}{|x-y|^{1+\alpha}}\dy\dx, \\
P_2 & = -\frac{\eps}{2} \iint_{\R\times \R}\left[ \vphid \left(\frac{|x|}{\eps^\beta}\right){w_\eps}(x,t) -\vphid \left(\frac{|y|}{\eps^\beta}\right){w_\eps}(y,t) \right] 
\frac{\left[ \vphid \left(\frac{|x|}{\eps^\beta}\right) -\vphid \left(\frac{|y|}{\eps^\beta}\right) \right] {w_\eps}(y,t)}{|x-y|^{1+\alpha}}\dy\dx  \\
&\quad +  \frac{\eps}{2} \iint_{\R\times \R}\left[\vphid\left(\frac{|x|}{\eps^\beta}\right) - \vphid\left(\frac{|y|}{\eps^\beta}\right)\right]^2 \frac{{w_\eps}(x,t){w_\eps}(y,t)}{|x-y|^{1+\alpha}}\dy\dx ,
\end{aligned}
\end{align*}
we have
\begin{align*}
\begin{aligned}
P & = -\frac{\eps}{2} \iint_{\R\times \R}\left[ \vphid \left(\frac{|x|}{\eps^\beta}\right){w_\eps}(x,t) -\vphid \left(\frac{|y|}{\eps^\beta}\right){w_\eps}(y,t) \right]^2 \frac{1}{|x-y|^{1+\alpha}}\dy\dx   \\
&\quad +  \frac{\eps}{2} \iint_{\R\times \R}\left[\vphid\left(\frac{|x|}{\eps^\beta}\right) - \vphid\left(\frac{|y|}{\eps^\beta}\right)\right]^2 \frac{{w_\eps}(x,t){w_\eps}(y,t)}{|x-y|^{1+\alpha}}\dy\dx ,
\end{aligned}
\end{align*}
which gives
\[
P\le \frac{\eps}{2} \iint_{\R\times \R}\left[\vphid\left(\frac{|x|}{\eps^\beta}\right) - \vphid\left(\frac{|y|}{\eps^\beta}\right)\right]^2 \frac{{w_\eps}(x,t){w_\eps}(y,t)}{|x-y|^{1+\alpha}}\dy\dx .
\]
Using \eqref{bdd-vs}, we have
\[
P\le C \eps \iint_{\R\times \R} \left[\vphid\left(\frac{|x|}{\eps^\beta}\right) - \vphid\left(\frac{|y|}{\eps^\beta}\right)\right]^2 \frac{1}{|x-y|^{1+\alpha}}\dy\dx .
\]
Since
\begin{align*}
\begin{aligned}
&  \iint_{\R\times \R} \left[\vphid\left(\frac{|x|}{\eps^\beta}\right) - \vphid\left(\frac{|y|}{\eps^\beta}\right)\right]^2 \frac{1}{|x-y|^{1+\alpha}}\dy\dx \\
& \quad =\left(\frac{1}{\delta\eps^\beta}\right)^{1+\alpha} \iint_{\R\times\R}\left[\vphi\left(\frac{|x|}{\delta\eps^\beta}\right) - \vphi \left(\frac{|y|}{\delta\eps^\beta}\right)\right]^2 \frac{1}{|x\delta^{-1}\eps^{-\beta}-y\delta^{-1}\eps^{-\beta}|^{1+\alpha}}\dy\dx \\
&\quad =\left(\frac{1}{\delta}\right)^{\alpha-1} \left(\frac{1}{\eps}\right)^{\beta(\alpha-1)} \iint_{\R\times\R} \frac{\left[\vphi\left(|x|\right) - \vphi\left(|y|\right)\right]^2}{|x-y|^{1+\alpha}}\dy\dx ,
\end{aligned}
\end{align*}
using $\beta=\frac{1}{\alpha-1}$, we have
\[
P\le C\left(\frac{1}{\delta}\right)^{\alpha-1}  \iint_{\R\times\R}  \underbrace{ \frac{\left[\vphi\left(|x|\right) - \vphi\left(|y|\right)\right]^2}{|x-y|^{1+\alpha}}  }_{=:J}  \dy\dx.
\]
Now, it remains to show that $J$ is integrable on $\R\times\R$. To this end, we separate the integral into several parts:
\begin{align*}
\begin{aligned}
 \iint  J&=\iint_{|x|\le1, |y|\le 2} J + \iint_{|x|\le1, |y|> 2} J +\iint_{|x|>1, |y|> 1} J +\iint_{1\le |x|\le 2, |y|\le1} J +\iint_{|x|> 2, |y|\le1} J    \\
 &=: I_1 +I_2 +I_3+I_4+I_5 .
\end{aligned}
\end{align*}
Using the smoothness of $\vphi$ and the boundedness of $\{|x|\le1, |y|\le 2\}$, we have
\[ 
I_1\le C \iint_{|x|\le1, |y|\le 2} \frac{1}{|x-y|^{\alpha-1}}\dy\dx ,
\]
which together with $1<\alpha<2$ implies $I_1<\infty$. \\
For $I_2$, observe that $|x-y| \geq |y| - |x| \geq |y| -1 \geq \frac{|y|}{2}$ for any $x, y$ with $|x|\le1, |y|> 2$, and $|\vphi(y) - \vphi(x)| \leq 2\|\vphi\|_\infty \le 2$. Thus, we have
\[
I_2\le C \int_{|x|\le1} \int_{ |y|> 2} \frac{1}{|y|^{1+\alpha}}\dy\dx <\infty .
\]
Since $\vphi(|x|) = \vphi(|y|) = 1$ for all $|x|, |y|\ge 1$ by \eqref{rough-phi}, we have $I_3 = 0$.\\
We use the same estimate as in $I_1$ to have $I_4 <\infty$.\\
Note that $J_5=J_2$ by symmetry.\\
Hence we complete the proof.
\end{proof}

\subsection{Conclusion}
It follows from Lemma \ref{lem:base} and Propositions \ref{prop:h1}, \ref{prop:h2}, \ref{prop:para} that for any $\delta\ge 4$ and $\eps>0$,
\begin{align*}
\begin{aligned}
\mathcal{H}'(t) &\le C \left( \sqrt{\delta\eps^\beta} + \left( \frac{1}{\delta}\right)^{3/2} +  \left( S_1 \left(\sqrt\delta\right) - u_+ \right) + \left( u_- - S_1 \left(-\sqrt\delta\right)  \right)  + \left( \frac{1}{\delta}\right)^{\alpha-1}  \right) \\
&\le C \underbrace{ \left( \sqrt{\delta\eps^\beta} +  \left( S_1 \left(\sqrt\delta\right) - u_+ \right) + \left( u_- - S_1 \left(-\sqrt\delta\right)  \right)  + \left( \frac{1}{\delta}\right)^{\alpha-1}  \right)}_{=:E(\eps, \delta)}.
\end{aligned}
\end{align*}
This, together with \eqref{defH}, implies
\begin{align*}
\begin{aligned}
& \int_{\{|x| \ge \delta\eps^\beta \}}\frac{ |\ueps(x+X(t),t)-\Seps(x)|^2 }{2} \, dx \le \mathcal{H}(t) =  \mathcal{H}(0) + \int_0^t  \mathcal{H}'(s) ds \\
&\quad \le  \mathcal{H}(0) +  CT E(\eps, \delta) .
\end{aligned}
\end{align*}
Moreover, since \eqref{bdd-vs} yields
\[
 \int_{\{|x| \le \delta\eps^\beta \}}\frac{ |\ueps(x+X(t),t)-\Seps(x)|^2 }{2} \, dx \le \frac{1}{2} \|v_\eps-\Seps\|_{L^\infty([0,T]\times\R)}^2  \int_{\{|x| \le \delta\eps^\beta \}} \, dx \le C \delta\eps^\beta ,
\]
we have
\begin{align*}
\begin{aligned}
&\int_\R \frac{ |\ueps(x+X(t),t)-\Seps(x)|^2 }{2} \, dx  \le \int_\R \frac{ | u_0(x)-\Seps(x)|^2 }{2} \, dx +  CT  \left( \delta\eps^\beta +E(\eps, \delta)  \right) .
\end{aligned}
\end{align*}
Therefore, 
\[
 \|\ueps(\cdot+X(t),t)-S_\eps\|_{L^2(\R)} \le \|u_0 -S_\eps\|_{L^2(\R)}   + C(T) \sqrt{\delta\eps^\beta +E(\eps, \delta) } .
\]
Then, using
\[
 \| S_\eps- S_0\|_{L^2(\R)} = \eps^{\beta/2} \| S_1 - S_0\|_{L^2(\R)}  ,
\]
we have
\begin{align*}
\begin{aligned}
\|\ueps(\cdot+X(t),t)-S_0\|_{L^2(\R)} &\le \|\ueps(\cdot+X(t),t)-S_\eps\|_{L^2(\R)} + \|S_\eps-S_0\|_{L^2(\R)} \\
&\le \|u_0 -S_\eps\|_{L^2(\R)}   +  C(T) \sqrt{\delta\eps^\beta +E(\eps, \delta) } + C \eps^{\beta/2} \\
&\le \|u_0 -S_0\|_{L^2(\R)} +\|S_\eps-S_0\|_{L^2(\R)}  +  C(T) \left(\sqrt{\delta\eps^\beta +E(\eps, \delta) }  + \eps^{\beta/2} \right)\\
&\le \|u_0 -S_0\|_{L^2(\R)}  + C(T) \left(\sqrt{\delta\eps^\beta +E(\eps, \delta) } + \eps^{\beta/2} \right).
\end{aligned}
\end{align*}
Therefore, for some constant $C(T)$,
\[
\|\ueps(\cdot+X(t),t)-S_0\|_{L^2(\R)}  \le \|u_0 -S_0\|_{L^2(\R)}  + C(T) \psi(\eps),
\]
where 
\[
\psi(\eps):=\inf_{\delta\ge 4} \left(\sqrt{\delta\eps^\beta +E(\eps, \delta) } + \eps^{\beta/2} \right) .
\]
This completes the proof.

\bibliography{Kang2019}

\begin{thebibliography}{10}

\bibitem{Alibaud}
N.~Alibaud.
\newblock Entropy formulation for fractal conservation laws.
\newblock {\em J. Evol. Equ.}, 7:145--175, 2007.

\bibitem{AIK}
N.~Alibaud, C.~Imbert, and G.~Karch.
\newblock Asymptotic properties of entropy solutions to fractal {Burgers}
  equation,.
\newblock {\em SIAM J. Math. Anal.}, 42:354--376, 2010.

\bibitem{BFW}
P.~Biler, T.~Funaki, and W.~Woyczynski.
\newblock Fractal {Burgers} equations.
\newblock {\em J. Differential Equations}, 148:9--46, 1998.

\bibitem{Chan}
C.~H. Chan and M.~Czubak.
\newblock Regularity of solutions for the critical {N-dimensional Burgers}
  equation,.
\newblock {\em Ann. Inst. H. Poincar\'e Anal. Non Lin\'eaire}, 27:471--501,
  2010.

\bibitem{Chmaj}
A.~Chmaj.
\newblock Existence of travelling waves in the fractional {Burgers} equation.
\newblock {\em Bull. Aus. Math. Soc.}, 97:102--109, 2018.

\bibitem{CV}
K.~Choi and A.~Vasseur.
\newblock Short-time stability of scalar viscous shocks in the inviscid limit
  by the relative entropy method.
\newblock {\em SIAM J. Math. Anal.}, 47:1405--1418, 2015.

\bibitem{Dafermos1}
C.~M. Dafermos.
\newblock The second law of thermodynamics and stability.
\newblock {\em Arch. Rational Mech. Anal.}, 70(2):167--179, 1979.

\bibitem{DiPerna}
R.~J. DiPerna.
\newblock Uniqueness of solutions to hyperbolic conservation laws.
\newblock {\em Indiana Univ. Math. J.}, 28(1):137--188, 1979.

\bibitem{Droniou}
J.~Droniou.
\newblock Vanishing non-local regularization of a scalar conservation law.
\newblock {\em Electron. J. Differential Equations}, 117:1--20, 2003.

\bibitem{DGV}
J.~Droniou, T.~Gallouet, and J.~Vovelle.
\newblock Global solution and smoothing effect for a non-local regularization
  of a hyperbolic equation.
\newblock {\em J. Evol. Equ.}, 3:499--521, 2002.

\bibitem{Kang19}
M.-J. Kang.
\newblock {$L^2$}-type contraction for shocks of scalar viscous conservation
  laws with strictly convex flux.
\newblock {\em https://arxiv.org/pdf/1901.02969.pdf}.

\bibitem{Kang18_KRM}
M.-J. Kang.
\newblock Non-contraction of intermediate admissible discontinuities for {3-D}
  planar isentropic magnetohydrodynamics.
\newblock {\em Kinet. Relat. Models}, 11(1):107--118, 2018.

\bibitem{Kang-V-NS17}
M.-J. Kang and A.~Vasseur.
\newblock {Contraction property for large perturbations of shocks of the
  barotropic {Navier-Stokes} system}.
\newblock {\em J. Eur. Math. Soc. (JEMS), To appear.
  https://arxiv.org/pdf/1712.07348.pdf}.

\bibitem{KV-unique19}
M.-J. Kang and A.~Vasseur.
\newblock {Uniqueness and stability of entropy shocks to the isentropic Euler
  system in a class of inviscid limits from a large family of Navier-Stokes
  systems}.
\newblock {\em https://arxiv.org/pdf/1902.01792.pdf}.

\bibitem{KVARMA}
M.-J. Kang and A.~Vasseur.
\newblock Criteria on contractions for entropic discontinuities of systems of
  conservation laws.
\newblock {\em Arch. Ration. Mech. Anal.}, 222(1):343--391, 2016.

\bibitem{Kang-V-1}
M.-J. Kang and A.~Vasseur.
\newblock {$L^2$}-contraction for shock waves of scalar viscous conservation
  laws.
\newblock {\em Annales de l'Institut Henri Poincar\'e (C) : Analyse non
  lin\'eaire}, 34(1):139Ð156, 2017.

\bibitem{KVW}
M.-J. Kang, A.~Vasseur, and Y.~Wang.
\newblock {$L^2$}-contraction for planar shock waves of multi-dimensional
  scalar viscous conservation laws.
\newblock {\em J. Differential Equations}, 267:2737--2791, 2019.

\bibitem{KMX}
G.~Karch, C.~Miao, and X.~Xu.
\newblock On convergence of solutions of fractal {Burgers equation} toward
  rarefaction waves,.
\newblock {\em SIAM J. Math. Anal.}, 39:1536--1549, 2008.

\bibitem{kiselev}
A.~Kiselev, F.~Nazarov, and R.~Shterenberg.
\newblock Blow up and regularity for fractal {Burgers} equation.
\newblock {\em Dyn. Partial Differ. Equ.}, 5:211--240, 2008.

\bibitem{Leger}
N.~Leger.
\newblock {$L^2$} stability estimates for shock solutions of scalar
  conservation laws using the relative entropy method.
\newblock {\em Arch. Ration. Mech. Anal.}, 199(3):761--778, 2011.

\bibitem{Serre_book}
D.~Serre.
\newblock {\em Systems of conservation laws I, II}.
\newblock Cambridge University Press, Cambridge, 1999.

\bibitem{Serre-Vasseur}
D.~Serre and A.~Vasseur.
\newblock {$L^2$}-type contraction for systems of conservation laws.
\newblock {\em J. \'Ec. polytech. Math.}, 1:1--28, 2014.

\bibitem{SV_16}
D.~Serre and A.~Vasseur.
\newblock About the relative entropy method for hyperbolic systems of
  conservation laws.
\newblock {\em Contemp. Math. AMS}, 658:237--248, 2016.

\bibitem{SV_16dcds}
D.~Serre and A.~Vasseur.
\newblock The relative entropy method for the stability of intermediate shock
  waves; the rich case.
\newblock {\em Discrete Contin. Dyn. Syst.}, 36(8):4569--4577, 2016.

\bibitem{Vasseur_Book}
A.~Vasseur.
\newblock Recent results on hydrodynamic limits.
\newblock In {\em Handbook of differential equations: evolutionary equations.
  {V}ol. {IV}}, Handb. Differ. Equ., pages 323--376. Elsevier/North-Holland,
  Amsterdam, 2008.

\bibitem{Vasseur-2013}
A.~Vasseur.
\newblock Relative entropy and contraction for extremal shocks of conservation
  laws up to a shift.
\newblock In {\em Recent advances in partial differential equations and
  applications}, volume 666 of {\em Contemp. Math.}, pages 385--404. Amer.
  Math. Soc., Providence, RI, 2016.

\bibitem{VW}
A.~Vasseur and Y.~Wang.
\newblock The inviscid limit to a contact discontinuity for the compressible
  navier-stokes-fourier system using the relative entropy method.
\newblock {\em SIAM J. Math. Anal.}, 47(6):4350--4359, 2015.

\end{thebibliography}

\end{document}